\documentclass[11pt]{amsart}
\usepackage[latin1]{inputenc}
\usepackage{amsmath,amsthm,amssymb}
\usepackage{amsfonts}
\usepackage{amsmath,amsthm,amssymb,amscd}
\usepackage{latexsym}
\usepackage{color}
\usepackage{graphicx}
\usepackage{mathrsfs}
\usepackage{comment}
\usepackage{cancel}

\makeatletter \@addtoreset{equation}{section} \makeatother

\newtheorem{theorem}{Theorem}[section]

\newtheorem{proposition}{Proposition}[section]
\newtheorem{lemma}{Lemma}[section]
\newtheorem{remark}{Remark}[section]

\newtheoremstyle{case}{}{}{}{}{}{\em:}{ }{}
\theoremstyle{case}

\numberwithin{subcase}{case}

\allowdisplaybreaks

\begin{document}


\title[Hardy double phase problems]
{A double phase problem involving Hardy potentials}

\author[Alessio Fiscella]{Alessio Fiscella}
\email{fiscella@ime.unicamp.br}

\address[Alessio Fiscella]{Departamento de Matem\'atica, Universidade Estadual de Campinas, IMECC\\
Rua S\'ergio Buarque de Holanda, 651, Campinas, SP CEP 13083-859 Brazil}

\subjclass[2010]{35J62; 35J92; 35J20.} \keywords{Double phase problems; Hardy potentials; variational methods.}

\date{July 16, 2020}
\maketitle


\begin{abstract} 
In this paper, we deal with the following double phase problem
$$
\left\{\begin{array}{ll}
-\mbox{div}\left(|\nabla u|^{p-2}\nabla u+a(x)|\nabla u|^{q-2}\nabla u\right)=
\gamma\left(\displaystyle\frac{|u|^{p-2}u}{|x|^p}+a(x)\displaystyle\frac{|u|^{q-2}u}{|x|^q}\right)+f(x,u) & \mbox{in } \Omega,\\
u=0 & \mbox{in } \partial\Omega,
\end{array}
\right.
$$
where $\Omega\subset\mathbb R^N$ is an open, bounded set with Lipschitz boundary, $0\in\Omega$, $N\geq2$, $1<p<q<N$, weight $a(\cdot)\geq0$, $\gamma$ is a real parameter and $f$ is a subcritical function. By variational method, we provide the existence of a non-trivial weak solution on the Musielak-Orlicz-Sobolev space $W^{1,\mathcal H}_0(\Omega)$, with modular function $\mathcal H(t,x)=t^p+a(x)t^q$. For this, we first introduce the Hardy inequalities for space $W^{1,\mathcal H}_0(\Omega)$, under suitable assumptions on $a(\cdot)$.

\end{abstract}

\maketitle

\section{Introduction}\label{sec:introduction}

In the present paper, we study the following problem 
\begin{equation}\label{P}
\left\{\begin{array}{ll}
-\mbox{div}\left(|\nabla u|^{p-2}\nabla u+a(x)|\nabla u|^{q-2}\nabla u\right)=
\gamma\left(\displaystyle\frac{|u|^{p-2}u}{|x|^p}+a(x)\displaystyle\frac{|u|^{q-2}u}{|x|^q}\right)+f(x,u) & \mbox{in } \Omega,\\
u=0 & \mbox{in } \partial\Omega,
\end{array}
\right.
\end{equation}
where $\Omega\subset\mathbb R^N$ is an open, bounded set with Lipschitz boundary, $0\in\Omega$, $N\geq2$, $\gamma$ is a real parameter, $1<p<q<N$ and
\begin{equation}\label{cruciale}
\frac{q}{p}<1+\frac{1}{N},\qquad a:\overline{\Omega}\to[0,\infty)\mbox{ {\em is Lipschitz continuous.}}
\end{equation}
Here, we assume that $f:\Omega\times\mathbb R\to\mathbb R$ is a {\em Carath\'{e}odory} function verifying
\begin{enumerate}
\item[$(f_1)$] {\em there exists an exponent $r\in (q,p^*)$, with the critical Sobolev exponent $p^*=Np/(N-p)$, such that for any $\varepsilon>0$ there exists $c_\varepsilon=c(\varepsilon)>0$ and
$$
|f(x,t)|\leq q\varepsilon\left|t\right|^{q-1} +r\delta_\varepsilon\left|t\right|^{r-1}
$$
holds for a.e. $x\in \Omega$ and any $t\in\mathbb{R}$};
\item[$(f_2)$]
{\em there exist $\theta\in (q,p^*)$, $c>0$ and $t_0\geq0$ such that}
$$c\leq\theta F(x,t)\leq tf(x,t)$$
{\em for a.e. $x\in\Omega$ and any }$|t|\geq t_0$, {\em where} $F(x,t)=\displaystyle\int^{t}_{0}f(x,\tau)d\tau$.
\end{enumerate}
The function $f(x,t)=\phi(x)\left(\theta t^{\theta-1}+r t^{r-1}\right)$, with $\phi\in L^\infty(\Omega)$ and $\phi>0$ a.e. in $\Omega$, verifies all assumptions $(f_1)-(f_2)$.

Problem \eqref{P} is driven by the so-called double phase operator, which switches between two different types of elliptic rates according to the coefficient $a(\cdot)$.
This kind of operator was introduced by Zhikov in \cite{Z1,Z2,Z3,ZK} in order to provide models for strongly anisotropic materials. Also, \eqref{P} falls into the class of problems driven by operators with non-standard growth conditions, according to Marcellini's definition given in \cite{M1,M2}. Following this direction, Mingione et al. prove different regularity results for minimizers of double phase functionals in \cite{BCM,CM1,CM2}. In \cite{CS}, Colasuonno and Squassina analyze the eigenvalue problem with Dirichlet boundary condition of the double phase operator. In particular, in \cite[Section 2]{CS} they provide the basic tools to solve variational problems like \eqref{P}, introducing the standard condition \eqref{cruciale}.
Recently, Mizuta and Shimomura study Hardy-Sobolev inequalities in the unit ball for double phase functionals in \cite{MS}.
While, for existence and multiplicity of solutions of nonlinear problems driven by the double phase operator, we refer to \cite{GLL,LD,PS}, with the help of variational techniques, and to \cite{GW1,GW2}, through a non-variational characterization.

Inspired by the above papers, we provide an existence result for \eqref{P} by variational method.
The main novelty, as well as the main difficulty, of problem \eqref{P} is the presence of a double phase Hardy potential. Indeed, such term is responsible of the lack of compactness of the Euler-Lagrange functional related to \eqref{P}.
In order to handle the double phase potential in \eqref{P}, our weight function $a:\overline{\Omega}\to[0,\infty)$ satisfies
\begin{enumerate}
\item[$(a)$] {\em $a(\lambda x)\leq a(x)$ for any $\lambda\in(0,1]$ and any $x\in\overline{\Omega}$.}
\end{enumerate}
A simple example of Lipschitz continuous function verifying $(a)$ is given by $a(x)=|x|$.
Also, we control parameter $\gamma$ with the Hardy constants
\begin{equation}\label{costante}
H_m:=\left(\frac{m}{N-m}\right)^{-m},
\end{equation}
when $m=p$ and $m=q$.
Thus, we are ready to introduce the main result of the paper.

\begin{theorem}\label{T1.1}
Let $\Omega\subset\mathbb R^N$ be an open, bounded set with Lipschitz boundary, $0\in\Omega$ and $N\geq2$.
Let $1<p<q<N$ and $a(\cdot)$ satisfy \eqref{cruciale} and $(a)$. Let $(f_1)-(f_2)$ hold true.
Then, for any $\gamma\in(-\infty,\min\{H_p,H_q\})$ problem \eqref{P} admits a non-trivial weak solution.
\end{theorem}

The proof of Theorem \ref{T1.1} is based on the application of the classical mountain pass theorem. Also, Theorem \ref{T1.1} generalizes \cite[Theorem 1.3]{LD}, where the authors consider problem \eqref{P} with $\gamma=0$. However, our situation with $\gamma\neq0$ is much more delicate than \cite{LD}, because of the lack of compactness, as well explained in Remark \ref{osservazione}.  

The paper is organized as follows.
In Section \ref{sec2}, we introduce the basic properties of the Musielak-Orlicz and Musielak-Orlicz-Sobolev spaces, including also the new Hardy inequalities, and we set the variational structure of problem \eqref{P}. In Section \ref{sec3}, we prove Theorem \ref{T1.1}.

\section{Preliminaries}\label{sec2}

The function $\mathcal H:\Omega\times[0,\infty)\to[0,\infty)$ defined as
$$\mathcal H(x,t):=t^p+a(x)t^q,\quad\mbox{for a.e. }x\in\Omega\mbox{ and for any }t\in[0,\infty),
$$
with $1<p<q$ and $0\leq a(\cdot)\in L^1(\Omega)$, is a generalized N-function (N stands for {\em nice}), according to the definition in \cite{D,M}, and satisfies the so called $(\Delta_2)$ condition, that is
$$\mathcal H(x,2t)\leq t^q\mathcal H(x,t),\quad\mbox{for a.e. }x\in\Omega\mbox{ and for any }t\in[0,\infty).
$$
Therefore, by \cite{M} we can define the Musielak-Orlicz space $L^{\mathcal H}(\Omega)$ as
$$L^{\mathcal H}(\Omega):=\left\{u:\Omega\to\mathbb R\mbox{ measurable}:\,\,\varrho_{\mathcal H}(u)<\infty\right\},
$$
endowed with the Luxemburg norm
$$\|u\|_{\mathcal H}:=\inf\left\{\lambda>0:\,\,\varrho_{\mathcal H}\left(\frac{u}{\lambda}\right)\leq1\right\},$$
where $\varrho_{\mathcal H}$ denotes the $\mathcal H$-modular function, set as
\begin{equation}\label{rhoh}
\varrho_{\mathcal H}(u):=\int_\Omega\mathcal H(x,|u|)dx=\int_\Omega\left(|u|^p+a(x)|u|^q\right)dx.
\end{equation}
By \cite{CS,D}, the space $L^{\mathcal H}(\Omega)$ is a separable, uniformly convex, Banach space.
While, by \cite[Proposition 2.1]{LD} we have the following relation between the norm $\|\cdot\|_{\mathcal H}$ and the $\mathcal H$-modular.

\begin{proposition}\label{P2.1}
Assume that $u\in L^{\mathcal H}(\Omega)$, $\{u_j\}_j\subset
L^{\mathcal H}(\Omega)$ and $c>0$. Then
\begin{itemize}
\item[$(i)$] for $u\neq0$, $\|u\|_{\mathcal H}=c\Leftrightarrow\varrho_{\mathcal H}\left(\frac{u}{c}\right)=1$;
\vspace{0.1cm}
\item[$(ii)$] $\|u\|_{\mathcal H}<1$ $(resp.=1,\,>1)$ $\Leftrightarrow \varrho_{\mathcal H}(u)<1$ $(resp.=1,\,>1)$;
\vspace{0.1cm}
\item[$(iii)$] $\|u\|_{\mathcal H}<1\Rightarrow\|u\|_{\mathcal H}^q\leq\varrho_{\mathcal H}(u)\leq\|u\|_{\mathcal H}^p$;
\vspace{0.1cm}
\item[$(iv)$] $\|u\|_{\mathcal H}>1\Rightarrow\|u\|_{\mathcal H}^p\leq\varrho_{\mathcal H}(u)\leq\|u\|_{\mathcal H}^q$;
\vspace{0.1cm}
\item[$(v)$] $\lim\limits_{j\rightarrow\infty}\|u_j\|_{\mathcal H}=0\,(\infty)\Leftrightarrow\lim\limits_{j\rightarrow\infty}\varrho_{\mathcal H}(u_j)=0\,(\infty)$.
\end{itemize}
\end{proposition}

The related Sobolev space $W^{1,\mathcal H}(\Omega)$ is defined by
$$W^{1,\mathcal H}(\Omega):=\left\{u\in L^{\mathcal H}(\Omega):\,\,|\nabla u|\in L^{\mathcal H}(\Omega)\right\},
$$
endowed with the norm
\begin{equation}\label{norma}
\|u\|_{1,\mathcal H}:=\|u\|_{\mathcal H}+\|\nabla u\|_{\mathcal H},
\end{equation}
where we write $\|\nabla u\|_{\mathcal H}=\||\nabla u|\|_{\mathcal H}$ to simplify the notation.
We denote by $W^{1,\mathcal H}_0(\Omega)$ the completion of $C^\infty_0(\Omega)$ in $W^{1,\mathcal H}(\Omega)$ which can be endowed with the norm
$$\|u\|:=\|\nabla u\|_{\mathcal H},
$$
equivalent to the norm set in \eqref{norma}, thanks to \cite[Proposition 2.18(iv)]{CS} whenever \eqref{cruciale} holds true. 

For any $m\in[1,\infty)$ we indicate with $L^m(\Omega)$ the usual Lebesgue space equipped with the norm $\|\cdot\|_m$.
Then, by \cite[Proposition 2.15(ii)-(iii)]{CS} we have the following embeddings.

\begin{proposition}\label{P2.3}
Let \eqref{cruciale} holds true. For any $m\in[1,p^*]$ there exists $C_m=C(N,p,q,m,\Omega)>0$ such that
$$\|u\|_m^m\leq C_m\|u\|^m
$$
for any $u\in W^{1,\mathcal H}_0(\Omega)$.
Moreover, the embedding $W^{1,\mathcal H}_0(\Omega)\hookrightarrow L^m(\Omega)$ is compact for any $m\in[1,p^*)$.
\end{proposition}

We point out that Proposition \ref{P2.3} holds true when $m=q$. Indeed, by \eqref{cruciale} and $q>1$, we have $N(q-p)<p<qp$ which implies that $q<p^*$.

While, we denote by $L^q_a(\Omega)$ the weighted space of all measurable functions $u:\Omega\to\mathbb R$ with the seminorm
$$
\|u\|_{q,a}:=\left(\int_\Omega a(x)|u|^q dx\right)^{1/q}<\infty.
$$
Using this further notation, in the next result we provide the Hardy inequalities for space $W^{1,\mathcal H}_0(\Omega)$. The proof of the lemma is inspired by \cite[Lemma 2.1]{GP}.

\begin{lemma}\label{L2.5}
Let \eqref{cruciale} and $(a)$ hold true. Then, for any $u\in W^{1,\mathcal H}_0(\Omega)$ we have
$$
\begin{aligned}
&H_p\|u\|_{H_p}^p\leq\|\nabla u\|_p^p,
\hspace{1.5cm}\mbox{with }\,\,\|u\|_{H_p}:=\int_\Omega\frac{|u|^p}{|x|^p}dx\\
&H_q\|u\|_{H_{q,a}}^q\leq\|\nabla u\|_{q,a}^q,
\hspace{1.1cm}\mbox{with }\,\,\|u\|_{H_{q,a}}:=\int_\Omega a(x)\frac{|u|^q}{|x|^q}dx,
\end{aligned}
$$
where $H_p$ and $H_q$ are given in \eqref{costante}.
\end{lemma}
\begin{proof}
By \cite[Lemma 2.1]{GP}, \eqref{rhoh} and Proposition \ref{P2.1}, we know that
$$
\|u\|_{H_p}^p\leq\left(\frac{p}{N-p}\right)^p\|\nabla u\|_p^p,
$$
for any $u\in W^{1,\mathcal H}_0(\Omega)$.
Now, taking inspiration from \cite[Lemma 2.1]{GP}, let $u\in C^\infty_0(\Omega)$.
Then, we have
$$|u(x)|^q=-\int_1^\infty\frac{d}{d\lambda}|u(\lambda x)|^qd\lambda=-q\int_1^\infty |u(\lambda x)|^{q-2}u(\lambda x)\nabla u(\lambda x)\cdot x\,d\lambda
$$
a.e. in $\mathbb R^N$.
Hence, by H\"older inequality, $(a)$ and trivially extending $a(\cdot)$ in the whole space $\mathbb R^N$
$$
\begin{aligned}
\int_\Omega a(x)\frac{|u(x)|^q}{|x|^q}dx&=\int_{\mathbb R^N}a(x)\frac{|u(x)|^q}{|x|^q}dx\\
&=-q\int_1^\infty\int_{\mathbb R^N}a(x)\frac{|u(\lambda x)|^{q-2}u(\lambda x)}{|x|^{q-1}}\nabla u(\lambda x)\cdot\frac{x}{|x|}dx\,d\lambda\\
&=-q\int_1^\infty\int_{\mathbb R^N}\frac{1}{\lambda^{N+1-q}}\,a\left(\frac{y}{\lambda}\right)\frac{|u(y)|^{q-2}u(y)}{|y|^{q-1}}\nabla u(y)\cdot\frac{y}{|y|}dy\,d\lambda\\
&\leq q\int_1^\infty\frac{d\lambda}{\lambda^{N+1-q}}\int_{\mathbb R^N}a(y)\frac{|u(y)|^{q-1}}{|y|^{q-1}}|\nabla u(y)|dy\\
&\leq\frac{q}{N-q}\left(\int_\Omega a(y)\frac{|u(y)|^q}{|y|^q}dy\right)^{(q-1)/q}\left(\int_\Omega a(y)|\nabla u(y)|^qdy\right)^{1/q}.
\end{aligned}
$$
From this, we obtain
$$\|u\|_{H_{q,a}}^q\leq\left(\frac{q}{N-q}\right)^q\|\nabla u\|_{q,a}^q,
$$
which holds true for any $u\in W^{1,\mathcal H}_0(\Omega)$ by density, \eqref{rhoh} and Proposition \ref{P2.1}.
\end{proof}

We are now ready to introduce the variational setting for problem \eqref{P}.
We say that a function $u\in W^{1,\mathcal H}_0(\Omega)$  is a weak solution of \eqref{P} if
$$
\int_\Omega\left(|\nabla u|^{p-2}+a(x)|\nabla u|^{q-2}\right)\nabla u\cdot\nabla\varphi dx
=\gamma\int_\Omega\left(\frac{|u|^{p-2}u}{|x|^p}+a(x)\frac{|u|^{q-2}u}{|x|^q}\right)\varphi dx+\int_\Omega f(x,u)\varphi dx,
$$
for any $\varphi\in W^{1,\mathcal H}_0(\Omega)$.
Clearly, the weak solutions of \eqref{P} are exactly the critical points of the Euler-Lagrange functional $J_\gamma:W^{1,\mathcal H}_0(\Omega)\to\mathbb R$, given by
$$J_\gamma(u):=\frac{1}{p}\|\nabla u\|_p^p+\frac{1}{q}\|\nabla u\|_{q,a}^q
-\gamma\left(\frac{1}{p}\|u\|_{H_p}^p+\frac{1}{q}\|u\|_{H_{q,a}}^q\right)
-\int_\Omega F(x,u)dx,
$$
which is well defined and of class $C^{1}$ on $W^{1,\mathcal H}_0(\Omega)$.

\section{Proof of Theorem \ref{T1.1}}\label{sec3}
Throughout the section we assume that $\Omega\subset\mathbb R^N$ is an open, bounded set with Lipschitz boundary, $0\in\Omega$, $N\geq2$, $1<p<q<N$, \eqref{cruciale} and $(a)$ hold true, without further mentioning. Also, we denote with $t^+=\max\{t,0\}$ and $t^-=\max\{-t,0\}$ respectively the positive and negative parts of a number $t\in\mathbb R$.

We recall that functional $J_\gamma:W^{1,\mathcal H}_0(\Omega)\to\mathbb R$ fulfills the Palais-Smale condition $(PS)$ if any sequence $\{u_j\}_j\subset W^{1,\mathcal H}_0(\Omega)$ satisfying
\begin{equation}\label{3.7}
\{J_\gamma(u_j)\}_j\mbox{ is bounded and }J'_\gamma(u_j)\to 0\mbox{ in }\left(W^{1,\mathcal H}_0(\Omega)\right)^{*}\mbox{ as }j\rightarrow\infty,
\end{equation}
possesses a convergent subsequence in $W^{1,\mathcal H}_0(\Omega)$.

The verification of the $(PS)$ condition for $J_\gamma$ is fairly delicate, considering the contribution of the double phase Hardy potential. Indeed, even if $W^{1,\mathcal H}_0(\Omega)\hookrightarrow L^p(\Omega,|x|^{-p})$ and $W^{1,\mathcal H}_0(\Omega)\hookrightarrow L^q(\Omega,a(x)|x|^{-q})$ by Lemma \ref{L2.5}, these embeddings are not compact. For this, we exploit a suitable tricky step analysis combined with the celebrated Br\'ezis and Lieb lemma in~\cite[Theorem 1]{BL}, which can be applied in $W^{1,\mathcal H}_0(\Omega)$ if we first prove the convergence $\nabla u_j(x)\to\nabla u(x)$ a.e. in $\Omega$, as $j\to\infty$.

\begin{proposition}\label{P3.3}
Let $(f_1)-(f_2)$ hold true. Then, for any $\gamma\in(-\infty,\min\{H_p,H_q\})$ the functional $J_\gamma$ verifies the $(PS)$ condition.
\end{proposition}
\begin{proof}
Let us fix $\gamma\in(-\infty,\min\{H_p,H_q\})$ and let $\{u_j\}_j\subset W^{1,\mathcal H}_0(\Omega)$ be a sequence satisfying \eqref{3.7}. 

We first show that $\{u_j\}_j$ is bounded in $W^{1,\mathcal H}_0(\Omega)$, arguing by contradiction. Then, going to a subsequence, still denoted by $\{u_j\}_j$, we have $\lim\limits_{j\to\infty}\|u_j\|=\infty$ and $\|u_j\|\geq1$ for any $j\geq n$, with $n\in\mathbb N$ sufficiently large.
Thus, according to $(f_2)$ and Lemma \ref{L2.5}, we get
\begin{align}\label{3.9}
J_\gamma(u_j)-\frac{1}{\theta}\langle J'_\gamma(u_j), u_j\rangle
=&\left(\frac{1}{p}-\frac{1}{\theta}\right)\|\nabla u_j\|_p^p
+\left(\frac{1}{q}-\frac{1}{\theta}\right)\|\nabla u_j\|_{q,a}^q
-\gamma\left(\frac{1}{p}-\frac{1}{\theta}\right)\|u_j\|_{H_p}^p
\nonumber\\
&-\gamma\left(\frac{1}{q}-\frac{1}{\theta}\right)\|u_j\|_{H_{q,a}}^q
-\int_{\Omega}\left[F(x,u_j)-\frac{1}{\theta}f(x,u_j)u_j\right]dx\nonumber\\
\geq&\left(\frac{1}{p}-\frac{1}{\theta}\right)\left(1-\frac{\gamma^+}{H_p}\right)\|\nabla u_j\|_p^p
+\left(\frac{1}{q}-\frac{1}{\theta}\right)\left(1-\frac{\gamma^+}{H_q}\right)\|\nabla u_j\|_{q,a}^q\\
&-\int_{\Omega_{t_0}}\left[F(x,u_j)-\frac{1}{\theta}f(x,u_j)u_j\right]^+dx\nonumber\\
\geq&\left(\frac{1}{q}-\frac{1}{\theta}\right)\left(1-\frac{\gamma^+}{\min\{H_p,H_q\}}\right)
\varrho_{\mathcal H}(\nabla u_j)
-D,\nonumber
\end{align}
since $\theta>q>p$ by $(f_2)$, where
$$
\Omega_{t_0}=\left\{x\in\Omega:\,\,|u_j(x)|\leq t_0\right\}\quad\mbox{ and }\quad
D=|\Omega|\sup_{x\in\Omega,|t|\leq t_0}\left[F(x,t)-\frac{1}{\theta}f(x,t)t\right]^+<\infty,
$$
with the last inequality which is consequence of $(f_1)$.
Thus, by \eqref{3.7} there exist $c_1$, $c_2>0$ such that \eqref{3.9} and Proposition \ref{P2.1} yield at once that as $j\rightarrow\infty$,
$$
c_1+c_2\|u_j\|+o(1)\geq
\left(\frac{1}{q}-\frac{1}{\theta}\right)\left(1-\frac{\gamma^+}{\min\{H_p,H_q\}}\right)
\|u_j\|^p
-D
$$
giving the desired contradiction, since $\theta>q>p>1$ and $\gamma<\min\{H_p,H_q\}$.

Hence, $\{u_j\}_j$ is bounded in $W^{1,\mathcal H}_0(\Omega)$.
By Propositions \ref{P2.1}-\ref{P2.3}, Lemma \ref{L2.5}, \cite[Theorem 4.9]{B} and the reflexivity of $W^{1,\mathcal H}_0(\Omega)$, there exists a subsequence, still denoted by $\{u_j\}_j$, and $u\in W^{1,\mathcal H}_0(\Omega)$ such that
\begin{equation}\label{3.10}
\begin{gathered}
u_j\rightharpoonup u\mbox{ in }W^{1,\mathcal H}_0(\Omega),\qquad
\nabla u_j\rightharpoonup\nabla u\mbox{ in }\left[L^{\mathcal H}(\Omega)\right]^N,\\
u_j\rightharpoonup u\mbox{ in }L^p(\Omega,|x|^{-p}),\quad
u_j\rightharpoonup u\mbox{ in }L^q(\Omega\setminus A,a(x)|x|^{-q}),\quad
\|u_j-u\|_{H_p}^p+\|u_j-u\|_{H_{q,a}}^q\to\ell,\\
u_j\to u\mbox{ in }L^m(\Omega),
\qquad u_j(x)\rightarrow u(x)\mbox{ a.e. in }\Omega,
\qquad |u_j(x)|\leq h(x)\mbox{ a.e. in }\Omega,
\end{gathered}
\end{equation}
as $j\to\infty$, with $m\in[1,p^*)$, $h\in L^q(\Omega)$ and $A$ is the nodal set of weight $a(\cdot)$, given by
$$
A:=\left\{x\in\Omega:\,\,a(x)=0\right\}.
$$
Indeed, since $a(\cdot)$ is a Lipschitz continuous function by \eqref{cruciale}, then $\Omega\setminus A$ is an open subset of $\mathbb R^N$. Also, $h\in L^q(\Omega)$ by Proposition \ref{P2.3} and \cite[Theorem 4.9]{B}, since $q<p^*$ by \eqref{cruciale}.

Now, we claim that 
\begin{equation}\label{claim}
\nabla u_j(x)\to\nabla u(x)\mbox{ a.e. in }\Omega,\mbox{ as }j\to\infty.
\end{equation}
Let $\varphi\in C^\infty(\mathbb R^N)$ be a cut-off function with $0\leq\varphi\leq1$, $\varphi\equiv1$ in $B(0,1/2)$ and $\varphi\equiv0$ in $B(0,1)$. Then, we define $\psi_R(x)=1-\varphi(x/R)$ for any $R>0$, so that $\psi_R\in C^\infty(\mathbb R^N)$ with $0\leq\psi_R\leq1$, $\psi_R\equiv1$ in $\mathbb R^N\setminus B(0,R)$, $\psi_R\equiv0$ in $B(0,R/2)$ and the sequence $\{\psi_R u_j\}_j$ is bounded in $W^{1,\mathcal H}_0(\Omega)$, thanks to Proposition \ref{P2.1}.
By simple calculation, for any $j\in\mathbb N$ we have
\begin{equation}\label{LG1}
\begin{aligned}
\langle J'_\gamma(u_j),\psi_R(u_j-u)\rangle=&
\int_\Omega\psi_R\left(|\nabla u_j|^{p-2}\nabla u_j+a(x)|\nabla u_j|^{q-2}\nabla u_j\right)\cdot(\nabla u_j-\nabla u)dx\\
&+\int_\Omega\left(|\nabla u_j|^{p-2}\nabla u_j+a(x)|\nabla u_j|^{q-2}\nabla u_j\right)\cdot\nabla \psi_R(u_j-u) dx\\
&-\gamma\int_\Omega\psi_R\left(\frac{|u_j|^{p-2}u_j}{|x|^p}+a(x)\frac{|u_j|^{q-2}u_j}{|x|^q}\right)(u_j-u)dx\\
&-\int_\Omega \psi_R f(x,u_j)(u_j-u)dx.
\end{aligned}
\end{equation}
Of course, all integrals in \eqref{LG1} are zero whenever $\overline{\Omega}\subset B(0,R/2)$, since $\psi_R\equiv0$ in $B(0,R/2)$. Thus, let us consider $R>0$ sufficiently small such that 
\begin{equation}\label{epsilon}
\left[\mathbb R^N\setminus B(0,R/2)\right]\cap\overline{\Omega}\neq\emptyset.
\end{equation}
By H\"older inequality, \eqref{3.10}, the facts that $\psi_R\in C^\infty(\mathbb R^N)$, $a(\cdot)$ is continuous in $\overline{\Omega}$ and $\{u_j\}_j$ is bounded in $W^{1,\mathcal H}_0(\Omega)$, we get
\begin{equation}\label{LG2}
\begin{aligned}
\int_\Omega&\left(|\nabla u_j|^{p-2}\nabla u_j+a(x)|\nabla u_j|^{q-2}\nabla u_j\right)\cdot\nabla \psi_R(u_j-u)dx\\
&\leq C\left(\|\nabla u_j\|_p^{p-1}\|u_j-u\|_p+\|\nabla u_j\|_{q,a}^{q-1}\|u_j-u\|_{q,a}\right)
\leq \widetilde{C}\left(\|u_j-u\|_p+\|u_j-u\|_q\right)\to0,
\end{aligned}
\end{equation}
as $j\to\infty$, for suitable $C$, $\widetilde{C}>0$.
Similarly, by considering also $(f_1)$ with $\varepsilon=1$, we obtain
\begin{equation}\label{LG3}
\begin{aligned}
\left|\int_\Omega \psi_R f(x,u_j)(u_j-u)dx\right|
&\leq \int_\Omega\left(q|u_j|^{q-1}+r\delta_1|u_j|^{r-1}\right)|u_j-u|dx\\
&\leq C\left(\|u_j-u\|_q+\|u_j-u\|_r\right)\to 0
\end{aligned}
\end{equation}
as $j\to\infty$, for a suitable $C>0$.
Furthermore, by \eqref{3.10} and \cite[Proposition~A.8]{AP}, considering that $a(\cdot)>0$ in $\Omega\setminus A$, we have
$$|u_j|^{p-2}u_j\rightharpoonup |u|^{p-2}u\,\mbox{  in }\,L^{p'}(\Omega,|x|^{-p}),\qquad
|u_j|^{q-2}u_j\rightharpoonup |u|^{q-2}u\,\mbox{  in }\,L^{q'}(\Omega\setminus A,a(x)|x|^{-q})
$$
so that
\begin{equation}\label{LG4}
\begin{aligned}
&\lim_{j\to\infty}\int_\Omega\psi_R\frac{|u_j|^{p-2}u_j}{|x|^p}u dx=
\int_\Omega\psi_R\frac{|u|^p}{|x|^p}dx,\\
&\lim_{j\to\infty}\int_\Omega\psi_R\, a(x)\frac{|u_j|^{q-2}u_j}{|x|^q}u dx=
\lim_{j\to\infty}\int_{\Omega\setminus A}\psi_R\, a(x)\frac{|u_j|^{q-2}u_j}{|x|^q}u dx=
\int_{\Omega\setminus A}\psi_R\,a(x)\frac{|u|^q}{|x|^q}dx\\
&\hspace{4.8cm}=\int_\Omega\psi_R\,a(x)\frac{|u|^q}{|x|^q}dx.
\end{aligned}
\end{equation}
While, by \eqref{3.10} it follows that
$$
\psi_R(x)\frac{|u_j(x)|^p}{|x|^p}\leq \left(\frac{2}{p}\right)^{p}|u_j(x)|^p\leq\left(\frac{2}{p}\right)^{p}h^p(x)
\quad\mbox{a.e in }\Omega\setminus B(0,R/2),
$$ 
so that, since $\psi_R\equiv0$ in $B(0,R/2)$, the dominated convergence theorem gives 
\begin{equation}\label{LG5}
\lim_{j\to\infty}\int_\Omega\psi_R\frac{|u_j|^p}{|x|^p}dx
=\lim_{j\to\infty}\int_{\Omega\setminus B(0,R/2)}\psi_R\frac{|u_j|^p}{|x|^p}dx
=\int_{\Omega\setminus B(0,R/2)}\psi_R\frac{|u|^p}{|x|^p}dx
=\int_\Omega\psi_R\frac{|u|^p}{|x|^p}dx.
\end{equation}
Similarly, by using also \eqref{cruciale}, for a suitable constant $L>0$ we get
$$
\psi_R(x)a(x)\frac{|u_j(x)|^q}{|x|^q}\leq L\left(\frac{2}{q}\right)^{q}h^q(x)
\quad\mbox{a.e in }\Omega\setminus B(0,R/2),
$$
which yields joint with the dominated convergence theorem 
\begin{equation}\label{LG5b}
\lim_{j\to\infty}\int_\Omega\psi_R\,a(x)\frac{|u_j|^q}{|x|^q}dx
=\int_\Omega\psi_R\,a(x)\frac{|u|^q}{|x|^q}dx.
\end{equation}
Thus, by \eqref{3.7}, \eqref{LG1}, \eqref{LG2}-\eqref{LG5b}, we obtain
$$\lim_{j\to\infty}\int_\Omega\psi_R\left(|\nabla u_j|^{p-2}\nabla u_j+a(x)|\nabla u_j|^{q-2}\nabla u_j\right)\cdot(\nabla u_j-\nabla u)dx=0.
$$
By H\"older inequality and being $\psi_R\leq 1$, we see that functional
$$
G:g\in\left[L^{\mathcal H}(\Omega)\right]^N\mapsto\int_\Omega\psi_R\left(|\nabla u|^{p-2}\nabla u+a(x)|\nabla u|^{q-2}\nabla u\right)\cdot g\,dx
$$
is linear and bounded.
Hence, by \eqref{3.10} we have
$$
\lim_{j\to\infty}\int_\Omega\psi_R\left(|\nabla u|^{p-2}\nabla u+a(x)|\nabla u|^{q-2}\nabla u\right)\cdot(\nabla u_j-\nabla u)dx=0,
$$
so that, denoting $\Omega_R:=\left\{x\in\Omega:\,\,|x|>R\right\}$ for any $R>0$, we get
\begin{equation}\label{LG6}
\begin{aligned}
&\lim_{j\to\infty}\int_{\Omega_R}\left[|\nabla u_j|^{p-2}\nabla u_j-|\nabla u|^{p-2}\nabla u
+a(x)\left(|\nabla u_j|^{q-2}\nabla u_j-|\nabla u|^{q-2}\nabla u\right)\right]\!\cdot\!(\nabla u_j-\nabla u)dx\\
&\leq\lim_{j\to\infty}\int_\Omega\psi_R\left[|\nabla u_j|^{p-2}\nabla u_j-|\nabla u|^{p-2}\nabla u
+a(x)\left(|\nabla u_j|^{q-2}\nabla u_j-|\nabla u|^{q-2}\nabla u\right)\right]\!\cdot\!(\nabla u_j-\nabla u)dx\\
&=0
\end{aligned}
\end{equation}
since $\psi_R\equiv1$ in $\mathbb R^N\setminus B(0,R)$.
Now, we recall the well known Simon inequalities, see \cite{S}, such that
\begin{equation}\label{simon}
|\xi-\eta|^m\leq
\begin{cases}
\kappa_m\,(|\xi|^{m-2}\xi-|\eta|^{m-2}\eta)\cdot(\xi-\eta), & \mbox{if $m\geq2$,}\\
\kappa_m\left[(|\xi|^{m-2}\xi-|\eta|^{m-2}\eta)\cdot(\xi-\eta)\right]^{m/2}\left(|\xi|^m+|\eta|^m\right)^{(2-m)/2}, & \mbox{if $1<m<2$,}
\end{cases}
\end{equation}
for any $\xi$, $\eta\in\mathbb R^N$, with $\kappa_m>0$ a suitable constant.
Therefore, if $p\geq2$ by \eqref{simon} we have
\begin{equation}\label{LG7}
\int_{\Omega_R}|\nabla u_j-\nabla u|^pdx
\leq \kappa_p\int_{\Omega_R}\left(|\nabla u_j|^{p-2}\nabla u_j-|\nabla u|^{p-2}\nabla u\right)\cdot(\nabla u_j-\nabla u)dx.
\end{equation}
While, if $1<p<2$ by \eqref{simon} and the H\"older inequality we obtain
\begin{equation}\label{LG8}
\begin{aligned}
&\int_{\Omega_R}|\nabla u_j-\nabla u|^pdx\\
&\,\,\,\,\leq \kappa_p
\int_{\Omega_R}\left[\left(|\nabla u_j|^{p-2}\nabla u_j-|\nabla u|^{p-2}\nabla u\right)\cdot(\nabla u_j-\nabla u)\right]^{p/2}
\left(|\nabla u_j|^p+|\nabla u|^p\right)^{(2-p)/2}dx\\
&\,\,\,\,\leq \kappa_p
\left[\int_{\Omega_R}\left(|\nabla u_j|^{p-2}\nabla u_j-|\nabla u|^{p-2}\nabla u\right)\cdot(\nabla u_j-\nabla u)dx\right]^{p/2}
\left(\|\nabla u_j\|_p^p+\|\nabla u\|_p^p\right)^{(2-p)/2}\\
&\,\,\,\,\leq \widetilde{\kappa_p}
\left[\int_{\Omega_R}\left(|\nabla u_j|^{p-2}\nabla u_j-|\nabla u|^{p-2}\nabla u\right)\cdot(\nabla u_j-\nabla u)dx\right]^{p/2}
\end{aligned}
\end{equation}
where the last inequality follows by the boundedness of $\{u_j\}_j$ in $W^{1,\mathcal H}_0(\Omega)$ and Proposition \ref{P2.1}, with a suitable new $\widetilde{\kappa_p}>0$.
Also, by convexity and since $a(x)\geq0$ a.e. in $\Omega$ by \eqref{cruciale}, we have
\begin{equation}\label{LG9}
a(x)\left(|\nabla u_j|^{q-2}\nabla u_j-|\nabla u|^{q-2}\nabla u\right)\cdot(\nabla u_j-\nabla u)
\geq0\mbox{ a.e. in }\Omega.
\end{equation}
Thus, combining \eqref{LG6}, \eqref{LG7}-\eqref{LG9} we prove that $\nabla u_j\to\nabla u$ in $[L^p(\Omega_R)]^N$ as $j\to\infty$, whenever $R>0$ satisfies \eqref{epsilon}. However, when $\overline{\Omega}\subset B(0,R/2)$ we have $\Omega_R=\emptyset$. Thus, for any $R>0$ the sequence $\nabla u_j\to\nabla u$ in $[L^p(\Omega_R)]^N$ as $j\to\infty$, and by diagonalization we prove claim \eqref{claim}.

Since the sequence $\{|\nabla u_j|^{p-2}\nabla u_j\}_j$ is bounded in $L^{p'}(\Omega)$, by \eqref{claim} we get
\begin{equation}\label{LG10}
\lim_{j\to\infty}\int_\Omega|\nabla u_j|^{p-2}\nabla u_j\cdot\nabla u\,dx=\|\nabla u\|_p^p.
\end{equation}
While, since $\{|\nabla u_j|^{q-2}\nabla u_j\}_j$ is bounded in $L^{q'}(\Omega\setminus A,a(x))$, by \eqref{claim} and \cite[Proposition A.8]{AP}
\begin{equation}\label{LG11}
\lim_{j\to\infty}\int_\Omega a(x)|\nabla u_j|^{q-2}\nabla u_j\cdot\nabla u\,dx=
\lim_{j\to\infty}\int_{\Omega\setminus A} a(x)|\nabla u_j|^{q-2}\nabla u_j\cdot\nabla u\,dx=\|\nabla u\|_{q,a}^q.
\end{equation}
Also, arguing as in \eqref{LG3} and \eqref{LG4}, we can prove
\begin{equation}\label{hardy}
\begin{aligned}
&\lim_{j\to\infty}\int_\Omega f(x,u_j)(u_j-u)dx=0,\\
&\lim_{j\to\infty}\int_\Omega\left(\frac{|u_j|^{p-2}u_j}{|x|^p}u+a(x)\frac{|u_j|^{q-2}u_j}{|x|^q}u\right)dx=
\|u\|_{H_p}^p+\|u\|_{H_{q,a}}^q.
\end{aligned}
\end{equation}
Furthermore, using \eqref{3.10}, \eqref{claim} and the Br\'ezis and Lieb lemma in~\cite[Theorem 1]{BL}, we obtain
\begin{equation}\label{Br}
\begin{aligned}
&\|\nabla u_j\|_p^p-\|\nabla u_j-\nabla u\|_p^p=\|\nabla u\|_p^p+o(1),
\hspace{0.6cm}\|\nabla u_j\|_{q,a}^q-\|\nabla u_j-\nabla u\|_{q,a}^q=\|\nabla u\|_{q,a}^q+o(1),\\
&\|u_j\|_{H_p}^p-\|u_j-u\|_{H_p}^p=\|u\|_{H_p}^p+o(1),
\hspace{1.15cm}\|u_j\|_{H_{q,a}}^q-\|u_j-u\|_{H_{q,a}}^q=\|u\|_{H_{q,a}}^q+o(1)
\end{aligned}
\end{equation}
as $j\to\infty$.  
Thus, by \eqref{3.7}, \eqref{LG10}, \eqref{LG11} and \eqref{hardy}, we get
\begin{equation}\label{3.12}
\begin{aligned}
o(1)=\langle J'_\gamma(u_j),u_j-u\rangle=&
\int_\Omega\left(|\nabla u_j|^{p-2}\nabla u_j+a(x)|\nabla u_j|^{q-2}\nabla u_j\right)\cdot(\nabla u_j-\nabla u)dx\\
&-\gamma\int_\Omega\left(\frac{|u_j|^{p-2}u_j}{|x|^p}+a(x)\frac{|u_j|^{q-2}u_j}{|x|^q}\right)(u_j-u)dx\\
&-\int_\Omega f(x,u_j)(u_j-u)dx\\
=&\|\nabla u_j\|_p^p-\|\nabla u\|_p^p
+\|\nabla u_j\|_{q,a}^p-\|\nabla u\|_{q,a}^p\\
&-\gamma\left(\|u_j\|_{H_p}^p-\|u\|_{H_p}^p+\|u_j\|_{H_{q,a}}^q-\|u\|_{H_{q,a}}^q\right)+o(1)
\end{aligned}
\end{equation}
as $j\to\infty$. 
Hence, by \eqref{Br} it follows that
\begin{equation}\label{formula}
\|\nabla u_j-\nabla u\|_p^p+\|\nabla u_j-\nabla u\|_{q,a}^q
=\gamma\left(\|u_j-u\|_{H_p}^p+\|u_j-u\|_{H_{q,a}}^q\right)+o(1)=\gamma\ell+o(1)
\end{equation}
as $j\to\infty$.
Now, assume for contradiction that $\ell>0$. Then, from Lemma \ref{L2.5}, \eqref{formula} and the fact that $\gamma<\min\{H_p,H_q\}$, we have
$$
\begin{aligned}
&\lim_{j\to\infty}\|\nabla u_j-\nabla u\|_p^p+\lim_{j\to\infty}\|\nabla u_j-\nabla u\|_{q,a}^q
\leq\gamma^+\left(\lim_{j\to\infty}\|u_j-u\|_{H_p}^p+\lim_{j\to\infty}\|u_j-u\|_{H_{q,a}}^q\right)\\
&\quad<\min\{H_p,H_q\}\left(\lim_{j\to\infty}\|u_j-u\|_{H_p}^p+\lim_{j\to\infty}\|u_j-u\|_{H_{q,a}}^q\right)
\leq\lim_{j\to\infty}\|\nabla u_j-\nabla u\|_p^p+\lim_{j\to\infty}\|\nabla u_j-\nabla u\|_{q,a}^q
\end{aligned}
$$
which is impossible.
Therefore $\ell=0$, so that by \eqref{formula} we have $\nabla u_j\to\nabla u$ in $\left[L^p(\Omega)\cap L_a^q(\Omega)\right]^N$ as $j\to\infty$, implying that $u_j\to u$ in $W^{1,\mathcal H}_0(\Omega)$ thanks to \eqref{rhoh} and Proposition \ref{P2.1}. This concludes the proof.

\end{proof}

Now, we complete the proof of Theorem \ref{T1.1}, proving first that functional $J_\gamma$ satisfies the geometric features of the mountain pass theorem.

\begin{lemma}\label{L3.1}
Let $(f_1)$ holds true. Then, for any $\gamma\in(-\infty,\min\{H_p,H_q\})$ there exist $\rho=\rho(\gamma)\in(0,1]$ and $\alpha=\alpha(\rho)>0$ such that $J_\gamma(u)\geq \alpha$ for any $u\in W^{1,\mathcal H}_0(\Omega)$, with $\|u\|=\rho$.
\end{lemma}

\begin{proof}
Let us fix $\gamma\in(-\infty,\min\{H_p,H_q\})$.
By $(f_1)$, for any $\varepsilon>0$ we have a $\delta_\varepsilon>0$ such that
\begin{equation}\label{3.1}
|F(x,t)|\leq\varepsilon|t|^q+\delta_\varepsilon|t|^r,\quad\mbox{for a.e. }x\in\Omega\mbox{ and any }t\in\mathbb R.
\end{equation}
Thus, by \eqref{3.1}, Lemma \ref{L2.5}, Propositions \ref{P2.1} and \ref{P2.3}, for any  $u\in W^{1,\mathcal H}_0(\Omega)$ with $\|u\|\leq1$, we obtain
\begin{align*}
J_\gamma(u)&\geq\frac{1}{p}\left(1-\frac{\gamma^+}{H_p}\right)\|\nabla u\|_p^p
+\frac{1}{q}\left(1-\frac{\gamma^+}{H_q}\right)\|\nabla u\|_{q,a}^q
-\varepsilon\|u\|_q^q-\delta_\varepsilon\|u\|_r^r\\
&\geq\frac{1}{q}\left(1-\frac{\gamma^+}{\min\{H_p,H_q\}}\right)\varrho_{\mathcal H}(\nabla u)
-\varepsilon C_q\|u\|^q-\delta_\varepsilon C_r\|u\|^r\\
&\geq\left[\frac{1}{q}\left(1-\frac{\gamma^+}{\min\{H_p,H_q\}}\right)-\varepsilon C_q\right]\|u\|^q-\delta_\varepsilon C_r\|u\|^r,
\end{align*}
since $q>p$ and $\gamma<\min\{H_p,H_q\}$.
Therefore, choosing $\varepsilon>0$ sufficiently small so that
$$\sigma_\varepsilon=\frac{1}{q}\left(1-\frac{\gamma^+}{\min\{H_p,H_q\}}\right)-\varepsilon C_q>0,
$$ 
for any $u\in W^{1,\mathcal H}_0(\Omega)$ with 
$\|u\|=\rho\in\big(0,\min\{1,[\sigma_\varepsilon/(2\delta_\varepsilon C_r)]^{1/(r-q)}\}\big]$, we get
$$
J_\gamma(u)\geq\left(\sigma_\varepsilon-\delta_\varepsilon C_r\rho^{r-q}\right)\rho^q:=\alpha>0.
$$
This completes the proof.
\end{proof}

\begin{lemma}\label{L3.2}
Let $(f_1)-(f_2)$ hold true. Then, for any $\gamma\in\mathbb R$ there exists $e\in W^{1,\mathcal H}_0(\Omega)$ such that $J_\gamma(e)<0$ and $\|e\|>1$.
\end{lemma}
\begin{proof}
Let us fix $\gamma\in\mathbb R$.
By $(f_1)$ and $(f_2)$, there exist $d_1>0$ and $d_2\geq0$ such that
\begin{equation}\label{3.4}
F(x,t)\geq d_1|t|^\theta-d_2\quad\mbox{for a.e. }x\in\Omega\mbox{ and any }t\in\mathbb R.
\end{equation}
Thus, if $\varphi\in W^{1,\mathcal H}_0(\Omega)$ with $\|\varphi\|=1$, then by Proposition \ref{P2.1} also $\varrho_{\mathcal H}(\nabla\varphi)=1$, so that by \eqref{3.4}, for any $t\geq1$ we have
$$
J_\gamma(t\varphi)\leq \frac{t^q}{p}-t^p\frac{\gamma^-}{p}\|\varphi\|_{H_p}^p-t^q\frac{\gamma^-}{q}\|\varphi\|_{H_{q,a}}^q
-t^\theta d_1\|\varphi\|_\theta^\theta-d_2|\Omega|.
$$
Since $\theta>q>p$ by $(f_2)$, passing to the limit as $t\rightarrow\infty$ we get $J_\gamma(t\varphi)\to-\infty$. Thus, the assertion follows by taking $e=t_{\infty}\varphi$, with $t_{\infty}$ sufficiently large.
\end{proof}

\begin{proof}[{\bf\textit{Proof of Theorem \ref{T1.1}}}]
Since $J_\gamma(0)=0$, by Proposition \ref{P3.3}, Lemmas \ref{L3.1}-\ref{L3.2} and the mountain pass theorem, we prove the existence of a non-trivial weak solution of \eqref{P}.
\end{proof}

We conclude this section with a result of independent interest, which shows how \eqref{claim} allows us to cover the complete situation in Theorem \ref{T1.1}, with $1<p<q<N$ and $\gamma\in(-\infty,\min\{H_p,H_q\})$. For this, we introduce the operator $L_\gamma:W^{1,\mathcal H}_0(\Omega)\to\left(W^{1,\mathcal H}_0(\Omega)\right)^*$, such that
$$\langle L_\gamma(u),v\rangle:=\int_\Omega\left(|\nabla u|^{p-2}+a(x)|\nabla u|^{q-2}\right)\nabla u\cdot\nabla v dx
-\gamma\int_\Omega\left(\frac{|u|^{p-2}u}{|x|^p}v+a(x)\frac{|u|^{q-2}u}{|x|^q}v\right)dx,
$$
for any $u$, $v\in W^{1,\mathcal H}_0(\Omega)$.

\begin{lemma}\label{L3.3}
Let $2\leq p<q<N$ and $\gamma\in(-\infty,\min\{H_p,H_q\}/\max\{\kappa_p,\kappa_q\})$, with $\kappa_p$ and $\kappa_q$ given by \eqref{simon}. Then, the operator $L_\gamma$ is a mapping of $(S)$ type, that is if $u_j\rightharpoonup u$ in $W^{1,\mathcal H}_0(\Omega)$ and 
\begin{equation}\label{L3.31}
\lim\limits_{j\to\infty}\langle L_\gamma(u_j)-L_\gamma(u),u_j-u\rangle=0,
\end{equation}
then $u_j\to u$ in $W^{1,\mathcal H}_0(\Omega)$.
\end{lemma}
\begin{proof}
Let us fix $2\leq p<q<N$ and $\gamma\in(-\infty,\min\{H_p,H_q\}/\max\{\kappa_p,\kappa_q\})$. Let $\{u_j\}_j$ be a sequence in $W^{1,\mathcal H}_0(\Omega)$ such that
$u_j\rightharpoonup u$ in $W^{1,\mathcal H}_0(\Omega)$ and \eqref{L3.31} holds true. Then, up to a subsequence $\{u_j\}_j$ is bounded in $W^{1,\mathcal H}_0(\Omega)$ and by Lemma \ref{L2.5} and \cite[Theorem 4.9]{B}, we obtain
$$
\|u_j-u\|_{H_p}^p+\|u_j-u\|_{H_{q,a}}^q\to\ell,
\qquad u_j(x)\rightarrow u(x)\mbox{ a.e. in }\Omega,
$$
as $j\to\infty$.
Thus, by \cite[Theorem 1]{BL} we get
\begin{equation}\label{L3.32}
\|u_j\|_{H_p}^p-\|u_j-u\|_{H_p}^p=\|u\|_{H_p}^p+o(1),\qquad\|u_j\|_{H_{q,a}}^q-\|u_j-u\|_{H_{q,a}}^q=\|u\|_{H_{q,a}}^q+o(1)
\end{equation}
as $j\to\infty$.  
While, by \eqref{simon} we have
\begin{equation}\label{L3.33}
\begin{aligned}
\int_{\Omega}&\left[\left(|\nabla u_j|^{p-2}\nabla u_j-|\nabla u|^{p-2}\nabla u\right)
+a(x)\left(|\nabla u_j|^{q-2}\nabla u_j-|\nabla u|^{q-2}\nabla u\right)\right]
\cdot(\nabla u_j-\nabla u)dx\\
&\geq\frac{1}{\max\{\kappa_p,\kappa_q\}}\left(\|u_j-u\|_p^p+\|u_j-u\|_{q,a}^q\right)
\end{aligned}
\end{equation}
for any $j\in\mathbb N$. Hence, combining \eqref{L3.31}-\eqref{L3.33}, as $j\to\infty$
$$
\frac{1}{\max\{\kappa_p,\kappa_q\}}\|\nabla u_j-\nabla u\|_p^p+\|\nabla u_j-\nabla u\|_{q,a}^q
=\gamma\left(\|u_j-u\|_{H_p}^p+\|u_j-u\|_{H_{q,a}}^q\right)+o(1)=\gamma\ell+o(1),
$$
which recalls \eqref{formula}, up to a constant. From this point, we can argue as in the end of the proof of Proposition \ref{P3.3}, proving that $u_j\to u$ in $W^{1,\mathcal H}_0(\Omega)$.
\end{proof}

\begin{remark}\label{osservazione}
{\em When $2\leq p<q<N$ and $\gamma\in(-\infty,K_{p,q}\min\{H_p,H_q\})$, with
$$
K_{p,q}:=\min\left\{1,\frac{1}{\max\{\kappa_p,\kappa_q\}}\right\},
$$
we can prove Proposition \ref{P3.3} arguing as in \cite[Lemma 5.1]{LD} and using Lemma \ref{L3.3} instead of \eqref{claim}.}
\end{remark}

\section*{Acknowledgments}

The author is member of {\em Gruppo Nazionale per l'Analisi Ma\-te\-ma\-ti\-ca, la Probabilit\`a e le loro Applicazioni} (GNAMPA) 
of the {\em Istituto Nazionale di Alta Matematica} (INdAM).
The author realized the manuscript within the auspices of the GNAMPA project titled {\em Equazioni alle derivate parziali: problemi e modelli} (Prot\_20191219-143223-545), of the FAPESP Project titled {\em Operators with non standard growth} (2019/23917-3), of the FAPESP Thematic Project titled {\em Systems and partial differential equations} (2019/02512-5) and of the CNPq Project titled {\em Variational methods for singular fractional problems} (3787749185990982).

\end{document}